\documentclass[11pt]{article}
\usepackage{amssymb}
\usepackage{mathrsfs}
\usepackage{latexsym,amsmath,amssymb,amsfonts,epsfig,graphicx,cite,psfrag}
\usepackage{eepic,color,colordvi,amscd}
\usepackage{ebezier}
\usepackage{verbatim}
\usepackage{subfigure}
\usepackage{enumitem}

\usepackage{amsthm}
\usepackage{comment}
\usepackage{color}
\theoremstyle{plain}
\newtheorem{theo}{Theorem}[section]

\newtheorem{lem}[theo]{Lemma}
\newtheorem{coro}[theo]{Corollary}

\theoremstyle{definition}

\theoremstyle{remark}

\def\qed{\hfill \rule{4pt}{7pt}}

\def\pf{\noindent {\it Proof.\quad}}

\def\Tr{\mathrm{Tr}}

\let\svthefootnote\thefootnote
\newcommand\blankfootnote[1]{%
	\let\thefootnote\relax\footnotetext{#1}%
	\let\thefootnote\svthefootnote%
}

\begin{document}
	\title{Monochromatic subgraphs in iterated triangulations
	}
	\author{Jie Ma\footnote{Email: jiema@ustc.edu.cn. Partially supported by NSFC grant 11622110.} ~ and ~Tianyun Tang\footnote{Email: tty123@mail.ustc.edu.cn.}\\ School of Mathematical Sciences\\University of Science and Technology of China\\Hefei, Anhui 230026, China\\
		\medskip\\
		Xingxing Yu\footnote{Email: yu@math.gatech.edu. Partially supported by NSF grant DMS-1600738.} \\
		School of Mathematics\\
		Georgia Institute of Technology\\
		Atlanta, GA 30332, USA}

	\date{}
	
	\maketitle

	\begin{abstract}

		
		For integers $n\ge 0$, an iterated triangulation $\Tr(n)$ is defined recursively as follows:
		$\Tr(0)$ is the plane triangulation on three vertices and, for $n\ge 1$, $\Tr(n)$ is the plane triangulation
		obtained from the plane triangulation $\Tr(n-1)$ by, for each inner face $F$ of $\Tr(n-1)$,
		adding inside $F$ a new vertex and three edges joining this new vertex to the three vertices incident with $F$.
		
		In this paper, we show that there exists a 2-edge-coloring of $\Tr(n)$ such that $\Tr(n)$ contains no monochromatic copy of the cycle $C_k$ for any $k\ge 5$. As a consequence,
		the answer to one of  two questions asked in \cite{ASTU19} is negative.
		We also determine the radius two graphs $H$ for which there exists $n$ such that every 2-edge-coloring of $\Tr(n)$
		contains a monochromatic copy of $H$, extending a result in \cite{ASTU19} for radius two trees.
		\blankfootnote{AMS Subject Classification: 05C55, 05C10, 05D10}
		\blankfootnote{Keywords: Triangulation, Planar unavoidable, Coloring, Monochromatic subgraph}
	\end{abstract}
	
	\section{Introduction}
	
	For graphs $G$ and $H$, we write $G \to H$ if, for any 2-edge-coloring of $G$, there is a monochromatic
	copy of $H$. Otherwise, we write $G \not\to H$.
	We say that $H$ is {\it planar unavoidable} if there
	exists a planar graph $G$ such that $G\to H$. This notion is introduced and studied in \cite{ASTU19}.
	
	Deciding if $G\to H$ is clearly equivalent to asking whether a graph $G$ admits a decomposition (i.e., an edge-decomposition)
	such that one of the two graphs in the decomposition contains the given graph $H$.
	The well-known Four Color Theorem \cite{AH77, AHK77} (also see \cite{RSST98}) implies that every planar graph admits
	a decomposition to two bipartite graphs; so planar unavoidable graphs must be bipartite.
	A result of Goncalves \cite{Go05} says that every planar graph admits a decomposition to two outer planar graph (although we have not seen a detailed proof); so planar unavoidable graphs must be also outer planar.
	There are a number of interesting results about decomposing planar graphs, see \cite{AA89, CCW86, Ha65, HMS96, Po90}.

	For any positive integer $n$, let $P_n$ denote the path on $n$ vertices, and $K_n$ denote the complete graph on
	$n$ vertices. For integer $n\ge 3$, we use $C_n$ to denote the cycle on $n$ vertices.
	It is shown in \cite{ASTU19} that $P_n$, $C_4$, and all trees with radius at most 2 are planar unavoidable.
	This is done by analyzing several sequences of graphs.

	In this paper, we investigate one such sequence -- the iterated triangulations,
	which is of particular interest as suggested in \cite{ASTU19}.
	Let $n\ge 0$ be an integer.
	An {\it iterated triangulation} $\Tr(n)$ is a plane graph defined as follows:
	$\Tr(0)\cong K_{3}$ is the plane triangulation with exactly two faces.
	For each $i\ge 0$, let $\Tr(i+1)$ be obtained from the plane triangulation $\Tr(i)$
	by adding a new vertex in each of the inner faces of $\Tr(i)$ and connecting this vertex with edges to the three vertices in the boundary of their respective face.
	The authors of \cite{ASTU19} asked whether for any planar unavoidable graph $H$ there is an integer $n$ such
	that $\Tr(n)\to H$.
	They also asked whether there exists an integer $k\ge 3$ such that the even cycle $C_{2k}$ is planar-unavoidable.
	
	Our first result indicates that a positive answer to one of the above questions implies a negative answer to the other.
	Let $H^+$ be the bipartite graph obtained by adding an edge to the unique 6-vertex tree with 4 leaves and 2 vertices of degree three.

	\begin{theo}\label{main}
		For all positive integer $n$, $\Tr(n) \not\to C_{k}$ for  $k \geq 5$,  $\Tr(n) \not\to H^+$, and  $\Tr(n) \not\to K_{2,3}$
	\end{theo}

	As another direct consequence, we see that if $B$ is a bipartite graph and $\Tr(n)\to B$ for some $n$ then every block of $B$ must be a $C_4$ or $K_2$. This can be used to
	characterize all radius two graphs $B$ for which there exists $n$ such that $\Tr(n)\to B$,
	generalizing a result in \cite{ASTU19} for radius two trees.
	To state this characterization, we need additional notations.
	A {\it flower} $F_{k}$ is a collection of $k$ copies of $C_4$s sharing a common vertex, which is called the {\it center}.
	A {\it jellyfish} $J_{k}$ is obtained from $F_k$ and a $k$-ary tree of radius two by
	identifying the center of $F_{k}$ with the root of the $k$-ary tree.
	A {\it bistar} $B_k$ is obtained from one $C_4$ and two disjoint $K_{1,k}$s by
	identifying the roots of the $K_{1,k}$s with two non-adjacent vertices of $C_4$, respectively.
	
	\begin{theo}\label{radius2}
		Let $L$ be a graph with radius two. Then there exists $n$ such that $\Tr(n)\to L$ if, and only if, $L$ is a subgraph of a jellyfish or bistar.
	\end{theo}

	We organize this paper as follows. In Section 2, we prove  $\Tr(n) \not\to C_{k}$ for  $k \geq 5$ and  $\Tr(n) \not\to H^+$ by finding a special edge coloring scheme for $\Tr(n)$.
	In Section 3, we complete the proof of Theorem~\ref{main} by using another edge coloring scheme on $\Tr(n)$.
	From Theorem \ref{main}, we can derive the following:
	if $L$ has radius 2 and $\Tr(n)\to L$ for some $n$, then $L$ is a subgraph of a jellyfish or bistar.
	Hence to prove Theorem~\ref{radius2}, it suffices to show that for any $k\ge 1$ there exists some $n$ such that $\Tr(n)\to J_k$ and $\Tr(n)\to B_k$.
	We prove the former statement in Section 4 and the latter one in Section 5 by showing that we can choose $n$ to be linear in $k$.

	\section{$H^+$ and $C_k$ for $k\geq 5$}
	
	In this section, we prove Theorem~\ref{main} for $H^+$ and $C_k$, with $k\ge 5$. First, we
	describe the 2-edge-coloring of $\Tr(n)$ that we will use.
	Let $\sigma: E(\Tr(n))\to \{0,1\}$ be defined inductively for all $n\geq 1$ as follows:
	\begin{itemize}
		\item [(i)] Fix an arbitrary triangle $T$ bounding an inner face of $\Tr(1)$, and let $\sigma(e)=0$ if $e\in E(T)$ and $\sigma(e)=1$
		if $e\in E(\Tr(1))\setminus E(T)$. 
		\item [(ii)] Suppose for some $1\le i<n$, we have defined $\sigma(e)$ for all $e\in E(\Tr(i))$.
		We extend $\sigma$ to $E(\Tr(i+1))$ as following.
		Let $x\in V(\Tr(i))\setminus V(\Tr(i-1))$ be arbitrary,  let $v_0v_1v_2v_0$ denote the triangle bounding
		the inner face of $\Tr(i-1)$ containing $x$, and  fix a labeling so that $\sigma(xv_1)=\sigma(xv_2)$.
		
		\item [(iii)] Let $x_j\in V(\Tr(i+1))\setminus V(\Tr(i))$ be such that
		$x_j$ is inside the face of $\Tr(i)$ bounded by the triangle $xv_jv_{j+1}x$, where $j=0,1,2$ and the subscripts are taken modulo 3.
		Define $\sigma(xv_0)=\sigma(x_0v_0)=\sigma(x_2v_0)=\sigma(x_jx)$ for all $j=0,1,2$, and
		$\sigma(xv_1)=\sigma(x_0v_1)=\sigma(x_1v_1)=\sigma(x_1v_2)=\sigma(x_2v_2)$.
	\end{itemize}

	We now proceed by a sequence of claims to show that $\sigma$ has no monochromatic $C_k$ for $k\geq 5$ nor monochromatic $H^+$,
	thereby proving $\Tr(n)\not\to C_{k}$ for $k\geq 5$ and $\Tr(n)\not\to H^+$.
	
	\begin{itemize}
		\item [(1)]  For $1\le i\le n$ and  $x\in V(\Tr(i))\setminus V(\Tr(i-1))$,
		$|\{\sigma(xv): v\in V(\Tr(i-1))\}|=2$.
	\end{itemize}
	We apply induction on $i$. The basis case $i=1$ follows from (i)
	above. So assume $2\le i\le n$.
	Let $v_0v_1v_2v_0$ be the triangle bounding the inner face of $\Tr(i-1)$ containing $x$.
	Without loss of generality assume that $v_0\in V(\Tr(i-1))\setminus V(\Tr(i-2))$.
	Let $v_1v_2v_3v_1$ denote the triangle bounding the face of $\Tr(i-2)$, with $v_0$
	inside $v_1v_2v_3v_1$.
	By induction hypothesis, $|\{\sigma(v_0v_k): k=1,2,3\}|=2$.
	
	Suppose $\sigma(v_0v_1)=\sigma(v_0v_2)$. Then by (ii) and (iii), $\sigma(xv_0)=\sigma(v_0v_3)$
	and $\sigma(xv_1)=\sigma(xv_2)=\sigma(v_0v_1)$.  So $|\{\sigma(xv_k): k= 0, 1,2\}|=2$.
	
	So assume $\sigma(v_0v_1)\neq\sigma(v_0v_2)$. By symmetry, we further assume
	$\sigma(v_0v_2)=\sigma(v_0v_3)$. Then by (ii) and (iii), we see that $\sigma(xv_1)=\sigma(xv_0)=\sigma(v_0v_1)$
	and $\sigma(xv_2)=\sigma(v_0v_2)$.
	So $\sigma(xv_2)\ne \sigma(xv_1)$ and hence, $|\{\sigma(xv_k): k=0,1,2\}|=2$. $\Box$

	\begin{itemize}
		\item [(2)]  Let $v_0v_1v_2v_0$ be a triangle bounding an inner face of $\Tr(i)$, where $0\le i<n$,
		let $v\in  V(\Tr(i+1))\setminus V(\Tr(i))$
		with $v$  inside $v_0v_1v_2v_0$. Then, for any $v_0w\in E(\Tr(n))$ with $w$ inside $v_0v_1v_2v_0$, $\sigma(v_0w)=\sigma(v_0v)$.
	\end{itemize}
	Let $v_0w\in E(\Tr(n))$ with $w$ inside $v_0v_1v_2v_0$. Then there exists $k\ge 0$ with $i+k+1\leq n$,
	such that $w\in V(\Tr(i+k+1))\setminus V(\Tr(i+k))$. We prove (2) by
	applying induction on $k$.
	The basis case is trivial because $k=0$ implies $w=v$.
	
	So assume $k \geq 1$. Let $v_0v_3v_4v_0$ be the triangle bounding  an inner face of
	$\Tr(i+k-1)$ with $w$ inside $v_0v_3v_4v_0$, and  let $v_5 \in
	V(\Tr(i+k))\setminus V(\Tr(i+k-1))$ that is inside $v_0v_3v_4v_0$.  By symmetry, assume $w$ is inside $v_0v_5v_4v_0$.
	By induction hypothesis, $\sigma(v_0v_5)$=$\sigma(v_0v)$.

	Suppose $\sigma(v_4v_5)=\sigma(v_0v_5)$. 
	Hence by (ii) and (iii), $\sigma(v_0w)=\sigma(wv_4)=\sigma(v_0v_5)$.
	Thus $\sigma(v_0w)=\sigma(v_0v)$.
	Now assume $\sigma(v_4v_5)\neq\sigma(v_0v_5)$. Then $\sigma(v_3v_5) = \sigma(v_0v_5)$ or $\sigma(v_3v_5)$=$\sigma(v_4v_5)$.
	It follows from (iii) that $\sigma(v_0w)=\sigma(v_0v_5)$. Hence, $\sigma(v_0w)=\sigma(v_0v)$. $\Box$

	\begin{itemize}
		\item [(3)] Let $v_0v_1v_2v_0$
		be a triangle bounding an inner face of $\Tr(i)$ with
		$0\le i\le n-2$, and let $v \in V(\Tr(i+1))\setminus V(\Tr(i))$ such that
		$v$ is inside $v_0v_1v_2v_0$
		and  $\sigma(vv_0)\ne \sigma(vv_1)=\sigma(vv_2)$. Then for any $vw\in
		E(\Tr(n))$ with $w$ inside $v_0v_1v_2v_0$,
		$\sigma(vw)=\sigma(vv_0)$.
	\end{itemize}
	To prove (3), let $\{w_0,w_1,w_2\}\subseteq V(\Tr(i+2))\setminus V(\Tr(i+1))$ such that $w_j$ is inside $vv_jv_{j+1}v$
	for $j=0,1,2$, with subscripts modulo 3. By (ii) and (iii), $\sigma(vw_0)=\sigma(vw_2)=\sigma(vw_1)=\sigma(vv_0)$.
	By (2), there exists some $j\in \{0,1,2\}$ with $\sigma(vw)=\sigma(vw_j)$.
	Hence, $\sigma(vw) = \sigma(vv_0)$. $\Box$
	
	\begin{itemize}
		\item [(4)] Let $v_0v_1v_2v_0$ be a triangle bounding an inner face of $\Tr(i)$, where $0\le i\le n-2$, and let
		$v\in V(\Tr(i+1))\setminus V(\Tr(i))$ such that $v$ is inside $v_0v_1v_2v_0$ and $\sigma(vv_0)\in \{\sigma(vv_1),\sigma(vv_2)\}$.
		Then for any $w\in (N(v)\cap N(v_0))\setminus \{v_1,v_2\}$,  $\sigma(wv_0) \neq \sigma(wv)$.
	\end{itemize}
	To prove (4), we may assume by symmetry and (1) that $\sigma(vv_2)\neq \sigma(vv_0)=\sigma(vv_1)$. Then $\sigma(wv_0)=\sigma(vv_0)$ by (2), and $\sigma(wv)=\sigma(vv_2)$ by (3). Hence, $\sigma(wv_0)\ne \sigma(wv)$.  $\Box$
	
	\begin{itemize}
		\item [(5)] Suppose $upv$ is a monochromatic path of length two in $\Tr(n)$ with $uv\in E(\Tr(i+1))$ and $p\in V(\Tr(n))\setminus V(\Tr(i+1))$.
		Then any monochromatic path in $\Tr(n)$ between $u$ and $v$ and of the color $\sigma(up)$ has length at most two.
	\end{itemize}
	Consider any monochromatic path $P=a_0a_1...a_r$ of the color $\sigma(up)$ with $a_0=v$ and $a_r=u$.
	First, suppose $uv\in E(\Tr(0))$. Let $\Tr(0)=uvwu$ and $x\in V(\Tr(1)) \setminus V(\Tr(0))$.
	By (2), $\sigma(ux)=\sigma(up)$ and $\sigma(vx)=\sigma(vp)$; so $\sigma(xu)=\sigma(xv)$. Thus, by (i),
	$\sigma(wx)=\sigma (wu)=\sigma(wv)\ne \sigma(xu)$.
	Let $v_0v_1\ldots v_n$ be a path in $\Tr(n)$ with $v_0=w$, $v_1=x$ and for $1\le i\le n$,
	$v_i\in V(\Tr(i))\setminus V(\Tr(i-1))$ is inside $v_{i-1}uvv_{i-1}$. By (ii) and (iii),
	$\sigma(v_iu)=\sigma(v_iv)=\sigma(vx)$ for $1\leq i\leq n$, and
	$\sigma(v_iv_{i+1})= \sigma(xw)$ for $0\le i\le n-1$.
	By planarity, $P$ is contained in the closed region bounded by $uvwu$.
	So either $P=uv$ or there exists some $1\leq k\leq r-1$ such that $a_k\in \{v_0,...,v_n\}$.
	We may assume the latter case occurs.
	If $\{a_{k-1}, a_{k+1}\}=\{u,v\}$, then $r=2$.
	Hence without loss of generality, let $a_{k-1}\notin \{u,v\}$.
	Then by (2) and (3), $\sigma(a_{k-1}a_k)=\sigma(v_iv_{i+1})\neq \sigma(pu)$ for $i\in \{0,1,...n-1\}$, a contradiction.
	Hence $r\le 2$.\footnote{We remark that this paragraph also shows that such $uv$ in $E(\Tr(0))$ cannot be in a monochromatic $C_4$.}

	Thus, we may assume $uv\notin E(\Tr(0))$. By symmetry, we may assume that
	$v\in V(\Tr(i+1)) \setminus V(\Tr(i))$ for some $0\le i <n$ and $v$ is inside the triangle $u_1u_2u_3u_1$ bounding an inner
	face of $\Tr(i)$  and $u_1=u$. By (4),  $\sigma(u_1v)\neq \sigma(u_2v) =\sigma(u_3v)$.
	
	If $a_1$ is inside $vu_2u_3v$ then there exists $1\le k<r$ such that $a_k$ is inside
	$vu_3u_2v$ and $a_{k+1}\in \{u_2, u_3\}$; so  by (2),
	$\sigma(a_ka_{k+1})= \sigma(vu_2)= \sigma(vu_3) \neq \sigma(u_1v)=\sigma(pu)$, a  contradiction.
	
	Therefore, suppose that $P\neq uv$, by symmetry, we may assume that $a_1$ is inside $u_1vu_2u_1$.
	Let $v_0=u_2$ and let $v_1v_2 \ldots v_{n-i-1}$ be the path in $\Tr(n)$ such that,
	for $1\le \ell\le n-i-1$, $v_\ell\in V(\Tr(i+\ell+1)) \setminus V(\Tr(i+\ell))$ is inside $u_1v_{\ell-1}vu_1$.
	
	By (ii) and (iii), $\sigma(v_\ell u_1)=\sigma(v_\ell v)=\sigma(u_1v)$ for
	$1\leq \ell\leq n-i-1$, and $\sigma(v_\ell v_{\ell+1})=\sigma(vu_2) \neq \sigma(vu_1)$
	for $0\leq \ell\leq n-i-2$.
	If $a_1$ is inside $v_\ell v_{\ell+1}vv_\ell$ for some $\ell$ with $0\le \ell\le n-i-2$, then exists $1\le k\le r$ such that
	$a_k$ is inside $v_\ell v_{\ell+1}vv_\ell$ and $a_{k+1}\in \{v_\ell,v_{\ell+1}\}$;
	so by (3) $\sigma(a_ka_{k+1}) = \sigma(v_\ell v_{\ell+1})$, a contradiction.
	So $a_1=v_\ell$ for some $\ell$ with $1\le \ell\le n-i-1$.
	Then as $\sigma(a_1a_2) = \sigma(u_1v)$ and by (3), we have $a_2=u_1$.
	Therefore, $r=2$, proving (5). $\Box$
	
	\medskip
	
	\begin{itemize}
		\item [(6)] If $C_k$ is monochromatic in $\Tr(n)$ then $k\leq 4$.
	\end{itemize}
	
	Let $C_k=a_1a_2\ldots a_ka_1$ be a monochromatic cycle in $\Tr(n)$.
	By (i), $E(C_k)\not \subseteq E(\Tr(0))$.
	So we may assume that there exists some $1\le i\le k$ such that $a_{i+1} \in V(\Tr(\ell+1))\setminus V(\Tr(\ell))$ is inside the triangle $a_iuva_i$
	which bounds an inner face of some $\Tr(\ell)$.
	We may further assume that $\ell\le n-2$, as otherwise, we could consider
	$\Tr(n+1)$ instead of $\Tr(n)$.\footnote{This is fair because $\Tr(n+1)\not\to C_k$ implies $\Tr(n)\not\to C_k$.}
	
	Suppose $\sigma(a_ia_{i+1})\in \{\sigma(a_{i+1}u),\sigma(a_{i+1}v)\}$. By symmetry, we may assume
	$\sigma(a_ia_{i+1})= \sigma(a_{i+1}u)$. Then $a_{i+2}=u$ by (3).
	Hence, by (5), any monochromatic path in $C_k$ between $a_i$ and $a_{i+2}=u$ has length at most 2.
	So $k\leq 4$.
	
	Thus, we may assume $\sigma(a_ia_{i+1})\notin \{\sigma(a_{i+1}u),\sigma(a_{i+1}v)\}$; hence, $\sigma(a_{i+1}u)=\sigma(a_{i+1}v)$.
	Let $w\in V(\Tr(\ell+2)) \setminus V(\Tr(\ell+1))$ be inside the triangle
	$a_iua_{i+1}a_i$. By (ii) and (iii), $\sigma(wa_i) = \sigma(wa_{i+1}) = \sigma(a_ia_{i+1})$.
	Hence, by (5), the monochromatic path $C_k-a_ia_{i+1}$ in $\Tr(n)$ of the color $\sigma(a_ia_{i+1})=\sigma(wa_i)$ has length at most 2; so $k=3$. $\Box$

	\begin{itemize}
		\item [(7)] There is no monochromatic $H^+$ in $\Tr(n)$.
	\end{itemize}
	
	Suppose that there is a monochromatic copy of $H^+$ on $\{v_i: 1\le i\le 6\}$ in which
	$v_1v_2v_3v_4v_1$ is a 4-cycle and $v_1v_5,v_2v_6$ are edges.
	If $v_1v_2\in E(\Tr(0))$, then $v_1v_2$ satisfies the conditions of (5) and by the footnote from the proof of (5),
	there is no monochromatic $C_4$ containing $v_1v_2$, a contradiction.
	So $v_1v_2\notin E(\Tr(0))$. By symmetry, we may assume that
	$v_2\in V(\Tr(i+1))\setminus V(\Tr(i))$ for some $i$ and that $v_1uwv_1$ is the triangle bounding the
	inner face of $\Tr(i)$ containing $v_2$.
	Again as before we may assume that $0\leq i\leq n-2$.
	
	If $\sigma(v_2u)=\sigma(v_2w)$, then there exists some $p\in V(\Tr(n))\setminus V(\Tr(i+1))$ such that
	$v_1pv_2$ has the same color as $\sigma(v_1v_2)$.
	But $v_1v_4v_3v_2$ is a monochromatic path of length 3 in $\Tr(n)$ between $v_1$ and $v_2$ and of the color $\sigma(v_1v_2)$, a contradiction to (5).
	
	Hence, $\sigma(v_1v_2)\in \{\sigma(v_2u), \sigma(v_2w)\}$ and by symmetry, we may assume $\sigma(v_1v_2)=\sigma(v_2u)$.
	Then by (1), $\sigma(v_1v_2)\neq\sigma(v_2w)$ and thus $\sigma(v_2v_3)=\sigma(v_2v_6)\ne \sigma(v_2w)$.
	This shows $w\notin \{v_3,v_6\}$. So there exists $y\in \{v_3,v_6\}\setminus \{u,w\}$.
	By (3), $\sigma(v_2y)=\sigma(v_2w)$, a contradiction.
	This completes the proof of this section.

	\section{Monochromatic $K_{2,3}$}
	
	In this section, we prove Theorem~\ref{main} for $K_{2,3}$ using a different coloring scheme on $\Tr(n)$ described below. Let $\sigma: E(\Tr(n))\to \{0,1\}$ be
	defined inductively as follows:
	\begin{itemize}
		\item [(i)] Fix a triangle $T$ bounding an inner face of $\Tr(1)$, and let $\sigma(e)=0$ if $e\in E(T)$ and $\sigma(e)=1$
		if $e\in E(\Tr(1))\setminus E(T)$. 
		\item [(ii)] Suppose for some $1\le i<n$, we have defined $\sigma(e)$ for all $e\in E(\Tr(i))$. We now extend $\sigma$ to $E(\Tr(i+1))$.
		Let $x\in V(\Tr(i))\setminus V(\Tr(i-1))$ be arbitrary,  let $v_0v_1v_2v_0$ denote the triangle bounding
		the inner face of $\Tr(i-1)$ containing $x$, with $v_0,v_1,v_2$ on the triangle in clockwise order, and  let
		$\sigma(xv_1)=\sigma(xv_2)$.
		\item [(iii)] Let $x_j\in V(\Tr(i+1))\setminus V(\Tr(i))$ such that
		$x_j$ is inside the face of $\Tr(i)$ bounded by the triangle $xv_jv_{j+1}x$,
		where $j=0,1,2$ and the subscripts are taken modulo 3. Define $\sigma(v_0x)=\sigma(v_0x_0)=\sigma(v_0x_2)
		=\sigma(xx_2) =\sigma(x_1v_1)$, and $\sigma(v_2x)=\sigma(v_2x_1)=\sigma(v_2x_2)=\sigma(xx_1)
		=\sigma(xx_0)=\sigma(x_0v_1)$.
	\end{itemize}
	
	
	Note that in (ii) we have $|\{\sigma(xv_j): j=0,1,2\}|=2$ and that in (iii) we have $\sigma(x_jv_j)\ne \sigma(x_jv_{j+1})$ for $j=0,1,2$.
	Hence, inductively, we have
	
	\begin{itemize}
		\item [(1)] For $1\le i\le n$ and $x\in V(\Tr(i))\setminus V(\Tr(i-1))$,
		$|\{\sigma(xv): v\in V(\Tr(i-1))\}|=2$.
		
		\item [(2)] If $x_1x_2x_3x_1$ is a triangle which bounds an inner face of $\Tr(i)$ for some $1\le i\le n-2$, and if  $x\in
		V(\Tr(n))\setminus V(\Tr(i+1))$ is inside $x_1x_2x_3x_1$ with $xx_1,xx_2\in E(\Tr(n))$, then $\sigma(xx_1)\neq \sigma(xx_2)$.
	\end{itemize}
	These two claims are straightforward so we omit their proofs.
	
	\begin{itemize}
		\item [(3)] For any $x_1x_2\in E(\Tr(n))$,  $|\{x\in
		N(x_1)\cap N(x_2): \sigma(xx_1)=\sigma(xx_2)=0\}|\le 2$ and  $|\{x\in N(x_1)\cap
		N(x_2): \sigma(xx_1)=\sigma(xx_2)=1\}|\le 2$.
	\end{itemize}
	First, suppose $x_1x_2\in E(\Tr(0))$. Then by (i) and (2),
	$|\{x\in N(x_1)\cap N(x_2): \sigma(xx_1)=\sigma(xx_2)=0\}|\le 1$ and  $|\{x\in N(x_1)\cap
	N(x_2): \sigma(xx_1)=\sigma(xx_2)=1\}|\le 1$.
	
	So we may  assume that $x_1vwx_1$ bounds an inner face of $\Tr(i)$ and $x_2\in V(\Tr(i+1))\setminus V(\Tr(i))$  inside $x_1vwx_1$.
	Let $v_1\in \Tr(i+2)$ be inside $x_1vx_2x_1$ and $w_1\in \Tr(i+2)$ be inside $x_1wx_2x_1$.  By (iii), $\sigma(w_1x_1)\neq \sigma(w_1x_2)$ or $\sigma(v_1x_1)\ne \sigma(v_1x_2)$.
	By (2), for any  $x\in V(\Tr(n))\setminus V(\Tr(i+2))$  inside $x_1vwx_1$ with $xx_1,xx_2\in E(\Tr(n))$, we have $\sigma(xx_1)\neq \sigma(xx_2)$.
	Hence, if (3) fails, then
	we may assume by symmetry between $w_1$ and $v_1$ that $\sigma(vx_1)=\sigma(vx_2)=\sigma(wx_1)=\sigma(wx_2)=\sigma(v_1x_1)=\sigma(v_1x_2)$, and $\sigma(w_1x_1)\neq \sigma(w_1x_2)$.
	Then, by (1), $\sigma(x_1x_2)\neq \sigma(x_2v)=\sigma(x_2w)$.
	Now by (iii), at least one of the two edges $v_1x_1$ and $v_1x_2$ has the same color as $x_1x_2$, a contradiction. This proves (3).
	
	\begin{itemize}
		\item [(4)] If $x_1x_2x_3x_4x_1$ is a 4-cycle in $\Tr(n)$, then $x_1x_3\in E(\Tr(n))$ or $x_2x_4\in E(\Tr(n))$.
	\end{itemize}
	We may assume that $\{x_1,x_2,x_3,x_4\}\subseteq V(\Tr(i+1))$ and $x_j\in V(\Tr(i+1))\setminus V(\Tr(i))$ for some $0\le i<n$ and $j\in [4]$.
	Let $uvwu$ be the triangle bounding an inner face of $\Tr(i)$ such that $x_j$ is inside it.
	Then $\{x_{j-1},x_{j+1}\}\subseteq \{u,v,w\}$, implying that $x_{j-1}x_{j+1}\in E(\Tr(n))$.   $\Box$.
	
	\begin{itemize}
		\item [(5)] There is no monochromatic $K_{2,3}$ in   $\Tr(n)$.
	\end{itemize}
	For, suppose $\Tr(n)$ has a  monochromatic copy of $K_{2,3}$ on $\{v_1,v_2,v_3,v_4,v_5\}$ with $v_4v_i, v_5v_i\in E(\Tr(n))$ for all $i=1,2,3$.
	Then $v_4v_5\notin E(\Tr(n))$ by (3) and, hence, it follows from (4) that $v_1v_2,v_2v_3,v_3v_1\in E(\Tr(n))$. By planarity,
	$v_1v_2v_3v_1$ bounds an inner face of $\Tr(i)$ for some $i$ with $1\le i<n$ and, by the symmetry between $v_4$ and $v_5$, we may
	assume  that $v_4$ is inside $v_1v_2v_3v_1$. Then $v_4\in V(\Tr(i+1))\setminus V(\Tr(i))$.
	However, this contradicts (1), as $\sigma(v_4v_1)=\sigma(v_4v_2)=\sigma(v_4v_3)$.
	We have completed the proof of Theorem \ref{main}. \qed

	\section{Monochromatic $J_k$}
	
	In this section we prove that $\Tr(100k)\to J_k$ holds for any positive integer $k$.
	
	We need the following result, which is Lemma 9 in \cite{ASTU19}.
	The original statement in \cite{ASTU19} states $\Tr(16)\to C_4$,
	but the same proof in \cite{ASTU19} actually gives the following stronger version.
	
	\begin{lem}\label{astu}
		If $xyzx$ bounds the outer face of $\Tr(16)$, then any 2-edge-coloring of $\Tr(16)$ gives a monochromatic $C_4$
		which intersects $\{x,y\}$.
	\end{lem}
	
	Note that if the triangle $xyzx$ bounds the outer face of $\Tr(n)$ and $v\in V(\Tr(1))\setminus V(\Tr(0))$ then the subgraph of $\Tr(n)$ contained in the
	closed disc bounded by $vxyv$ is isomorphic to $\Tr(n-1)$.  Hence, the following is an easy consequence of Lemma~\ref{astu}.

	\begin{coro}\label{l1}
		If $xyzx$ bounds the outer face of $\Tr(17)$  then any 2-edge-coloring of $\Tr(17)$ gives  a monochromatic $C_4$ which intersects $\{x,y\}$ and avoids $z$.
	\end{coro}

	\begin{lem}\label{l2}
		For any positive integer $k$,  $\Tr(38k)$$\rightarrow F_k$
	\end{lem}
	\begin{proof}
		Let $\sigma: E(\Tr(38k))\to \{0,1\}$ be an arbitrary 2-edge coloring.
		Let $uvwu$ be the triangle bounding the outer face of $\Tr(38k)$. Let $x_0:=w$ and, for $1\le \ell\le 2k$, let
		$x_l\in V(\Tr(\ell))\setminus V(\Tr(\ell-1))$ such that $x_\ell$ is inside $x_{\ell-1}uvx_{\ell-1}$.
		Let $y_{i,0}:=x_i$ for $i\in \{0,1, \ldots, 2k-1\}$ and, for $\ell\in \{1, \ldots , 36k\}$,
		let  $y_{i,\ell}\in  V(\Tr(i+1+\ell))\setminus V(\Tr(i+\ell))$ such that $y_{i,\ell}$ is inside $y_{i,\ell-1}ux_{i+1}y_{i,\ell-1}$.

		Suppose for each $0\leq i\leq 2k-1$ there exists a monochromatic $C_4$ inside $x_iux_{i+1}x_i$ that contains $u$
		and avoids $x_i$. By pigeonhole principle, at least $k$ of these $C_4$s are of the same color, which form a
		monochromatic $F_k$ centered at $u$.

		Hence, we may assume that there exists some $i\in \{0,1,\ldots, 2k-1\}$ such that no monochromatic
		$C_4$ inside $x_iux_{i+1}x_i$ contains $u$ and avoids $x_i$.
		Since $i\le 2k-1$,  $x_iux_{i+1}x_i$ bounds the outer face of a $\Tr(36k)$ that is contained in $\Tr(38k)$.
		
		Now for each $h\in \{0, 1, \ldots, 2k-1\}$, we view the closed region bounded by $ux_{i+1}y_{i,18h}u$ as a $\Tr(17)$.
		Note that these copies of $\Tr(17)$ share $u, x_{i+1}$ as the only common vertices.
		Taking $y_{i,18h}$ to be the vertex $z$ in Corollary \ref{l1}, we conclude from Corollary \ref{l1} that
		there is a monochromatic $C_4$ in the $\Tr(17)$ bounded by $ux_{i+1}y_{i,18h}u$ such that $x_{i+1}\in V(G_h)$ and $\{u,y_{i,18h}\}\cap V(G_h)=\emptyset$.
		By pigeonhole principle, at least $k$ of these $C_4$s are of the same color,
		which clearly form a monochromatic $F_k$ centered at $x_{i+1}$.
	\end{proof}


	\begin{lem}\label{l3}
		Let $k$ be a positive integer and let $uvwu$ bound the outer face of $\Tr(9k+2)$.
		Suppose $\sigma: E(\Tr(9k+2))\to \{0,1\}$ is a 2-edge-coloring such that $|\{\sigma(ux): x\in V(\Tr(9k+2))\}|=1$ and
		there is no monochromatic $C_4$ containing $u$. Then
		$\Tr(9k+2)$ contains monochromatic $J_k$ centered at $v$.
	\end{lem}
	\begin{proof} Without loss of generality, assume $\sigma(uv)=0$. Then $\sigma(uy)=0$ for all $y\in N(u)$.
		Let $x_0:=w$ and, for $1\le i\le 8k+1$, let $x_i\in V(\Tr(i))\setminus V(\Tr(i-1))$ such that $x_i$ is inside $x_{i-1}uvx_{i-1}$.
		Since no monochromatic  $C_4$ in $\Tr(9k+2)$ contains $u$, there is at most one $i\in \{0,1,2,...8k+1\}$ such that
		$\sigma(x_iv)=0$. Hence, there exists $i\in \{0,1,\ldots 4k+2\}$ such that  $\sigma(vx_j)=1$ for $j\in \{i,i+1\ldots i+4k-1 \}$.
		We now make the following claim.
		
		
		\begin{itemize}[leftmargin=14mm]
			\item[Claim.] The subgraph of $\Tr(9k+2)$ contained in the closed disc bounded by $vx_i\ldots x_{i+3k-1}v$ has
			a monochromatic $F_k$ of color 1 and centered at $v$, which we denote by $F_v$.
		\end{itemize}
		To show this, it suffices to show that for each $r$ with $0\le r\le k-1$,
		the subgraph $T_r$ of $\Tr(9k+2))$ inside $vx_{i+3r}x_{i+3r+1}x_{i+3r+2}v$ (inclusive)
		contains a monochromatic $C_4$ of color 1 and containing $v$, as the union of such $C_4$ is an $F_k$ centered at $v$.
		So fix an arbitrary $r$, with $0\le r\le k-1$.
		Note that  $\sigma(x_{i+3r}x_{i+3r+1})=1$ or
		$\sigma(x_{i+3r+1}x_{i+3r+2})=1$, for $0\leq r\leq k-1$; for, otherwise, $x_{i+3r}x_{i+3r+1}x_{i+3r+2}ux_{i+3r}$ is a monochromatic $C_4$
		of color 0 and containing $u$, a contradiction. Without loss of generality, assume
		$\sigma(x_{i+3r}x_{i+3r+1})=1$.
		
		Let $y\in V(\Tr(i+3r+2))\setminus V(\Tr(i+3r+1))$ such that $y$ is inside
		$x_{i+3r}x_{i+3r+1}vx_{i+3r}$.
		If there are two edges in $\{yx_{i+3r},yx_{i+3r+1},yv\}$ of color 0,
		then one can easily find a monochromatic $C_4$ of color $0$ and containing $u$, a contradiction.
		Hence, at least two of $\{\sigma(yx_{i+3r}),\sigma(yx_{i+3r+1}),\sigma(yv)\}$ are 1.
		So $\{y,x_{i+3r},x_{i+3r+1},v\}$ induces a subgraph which contains a monochromatic $C_4$ of color 1.
		This proves the claim. $\Box$

		\medskip
		
		
		
		Note that for $i+3k\le r\le i+4k-1$, $ux_rx_{r-1}u$ bounds
		the outer face of a $\Tr(k+1)$. Let $z_{r,0}:=x_{r-1}$ and, for $r \in \{i+3k,i+3k+1,\ldots i+4k-1\}$
		and $\ell\in \{1,2,\ldots k\}$, let $z_{r,\ell}\in V(Tr(r+\ell))\setminus V(Tr(r+\ell-1))$
		such that $z_{r,\ell}$ is inside $z_{r,\ell-1}x_ruz_{r,\ell-1}$.
		Because $\sigma(uz_{r,j})=0$ (by assumption) and $\Tr(9k+2)$ has no monochromatic $C_4$ containing  $u$, there is
		at most one $y\in \{z_{r,1},z_{r,2},\ldots,z_{r,k}\}$ such that $\sigma(yx_r)=0$. So
		there exists $k-1$ vertices in $\{z_{r,1},..., z_{r,k}\}$ which together with
		$x_rv$ form a monochromatic $K_{1,k}$ of color 1  centered at $x_r$, which we denote by
		$H_r$.
		Now $H_{i+3k}, H_{i+3k+1}, \ldots, H_{i+4k-1}$ form a monochromatic $k$-ary radius 2 tree rooted at $v$ of color 1.
		This radius 2 tree and $F_v$ form a monochromatic $J_k$ of color 1, completing the proof of Lemma \ref{l3}.
	\end{proof}

	Now we are ready to prove the main result of this section, that is $\Tr(100k)\to J_k$.
	Let $\sigma: E(\Tr(100k))\to \{0,1\}$ be arbitrary. We show that $\sigma$ always contains a monochromatic $J_k$.
	By Lemma \ref{l2}, $\Tr(76k)$ contains monochromatic copy of $F_{2k}$, say $F$, and, without loss of generality,
	assume it is of color 1.
	Let the $C_4$s in $F$ be $xa_{i,1}a_{i,2}a_{i,3}x$ for $i \in [2k]$.
	For $i\in [2k]$, let $b_i\in V(\Tr(76k+1))\setminus V(\Tr(76k))$ such that $b_i$ is inside $xa_{i,1}a_{i,2}a_{i,3}x$ and
	$a_{i,1}a_{i,2}b_ia_{i,1}$ bounds an inner face of $\Tr(76k+1)$.
	Let $A_i$ be the family of all vertices $a\in N(a_{i,1})$ inside $a_{i,1}a_{i,2}b_ia_{i,1}$ and satisfying $\sigma(aa_{i,1})=1$.
	
	\begin{itemize}
		\item [(1)]  There exists some $i\in \{k+1,k+2,...2k\}$ such that $|A_i|<k$.
	\end{itemize}
	Otherwise, suppose  $|A_i|\geq k$ for $i\in \{k+1,k+2,...2k\}$.  Then let
	$Z_{i}:=\{z_{i,1},z_{i,2},...z_{i,k-1}\}\subseteq A_i$. Now, for each $i\in \{k+1, \ldots, 2k\}$,
	$\{x,a_{i,1}\}\cup Z_i$ induces a graph containing a monochromatic $K_{1,k}$.
	Those $K_{1,k}$s form a monochromatic radius-two $k$-ary tree of color 1 and rooted at $x$, which we denote by $T_x$.
	Now $F_k\cup T_x$ is a monochromatic $J_k$. $\Box$
	
	\medskip

	Let $u:=a_{i,1}$. By (1), there exists an edge $vw\in \Tr(78k)$ such that $uvwu$ bounds an inner face of $\Tr(78k)$
	and $\sigma(uy)=0$ for any $y\in N(u)$ in the closed disc bounded by $uvwu$.
	
	Let $G$ be the subgraph of $\Tr(n)$ contained in the closed disc bounded by $uvwu$.
	Clearly $G$ is isomorphic to a copy of $\Tr(22k)$.
	In the rest of the proof, we should only discuss the graph $G$ and all $\Tr(i)$ will be referred to this copy of $\Tr(22k)$.
	Let $x_0:=w$ and for $i\in [4k]$, let $x_{i}\in V(\Tr(i))\setminus V(\Tr(i-1))$ such that $x_{i}$ is inside $ux_{i-1}vu$.
	
	\begin{itemize}
		\item [(2)] $G$ contains a monochromatic copy of $F_k$, say $F'$, which has color 0 and center $u$ and is disjoint from the union of closed regions bounded by $ux_ix_{i+1}u$ over all $0\leq i\leq 2k-1$.
	\end{itemize}
	If for each $i\in \{k,k+1,...,2k-1\}$ there exists a monochromatic $C_4$ inside $ux_{2i}x_{2i+1}u$ and containing $u$,
	then these $k$ monochromatic $C_4$s of color 0 form a desired monochromatic $F_k$ centered at $u$ and thus (2) holds.
	Otherwise, since $ux_{2i}x_{2i+1}u$ bounds the outer face of a $\Tr(9k+2)$,
	it follows from Lemma~\ref{l3} that  there exists a monochromatic $J_k$ in $G$. $\Box$
	
	\medskip
	
	For $j\in \{0,1,...,2k-1\}$, let
	$B_j$ be the family of all vertices $x\in N(x_j)$ inside $ux_jx_{j+1}u$ and satisfying $\sigma(xx_j)=0.$
	
	\begin{itemize}
		\item [(3)] There exists some $j\in \{0,1,..., k-1\}$ such that $|B_j|<k$.
	\end{itemize}
	Suppose on the contrary that there exist subsets $Z_j\subseteq B_j$ of size $k$ for all $j\in \{0,1,...,k-1\}$.
	Then each $Z_j\cup \{u, x_j\}$ induces a graph containing a monochromatic $K_{1,k}$ which is centered at $x_j$ and has color 0.
	These $K_{1,k}$s together with $F'$ form a monochromatic $J_k$ of color 0.
	This proves (3). $\Box$
	
	\medskip
	
	Let $p_0:=x_{j+1}$ and for $1\leq \ell\leq 4k$, let
	$p_\ell\in V(\Tr(j+\ell+1))\setminus V(\Tr(j+\ell))$ such that $p_\ell$ is inside $ux_jp_{\ell-1}u$.
	By (3), there exists some $0\leq \ell\leq 4k-1$ such that $\sigma(zx_j)=1$ for any $z\in N(x_j)$ in the closed disc bounded by $x_jp_\ell p_{\ell+1}x_j$.
	
	\begin{itemize}
		\item [(4)]
		There is a monochromatic $F_k$ inside $x_jp_\ell p_{\ell+1}x_j$, say $F''$, with color 1 and center $x_j$.
	\end{itemize}
	Let $z_0:=p_{\ell+1}$ and for $s\in [2k]$, let $z_s\in V(\Tr(j+\ell+s+2))\setminus V(\Tr(j+\ell+s+1))$
	such that $z_s$ is inside $x_jz_{s-1}p_\ell x_j$.
	Note that each $x_jz_{2s}z_{2s+1}x_j$ bounds a $\Tr(9k+3)$.
	If for each $s\in [k]$ there exists a monochromatic $C_4$ of color 1 inside $x_jz_{2s}z_{2s+1}x_j$ and containing $x_j$,
	then these monochromatic copies of $C_4$ form the desired monochromatic $F_k$ centered at $x_j$.
	Otherwise, it follows from Lemma \ref{l3} that there exists a monochromatic $J_k$. $\Box$

	\medskip
	
	As $|B_j|<k$, there exists a subset $A\subseteq \{p_1,p_2,...,p_{4k}\}$ of size $2k$ such that $\sigma(\alpha x_j)=1$ for each $\alpha\in A$
	and moreover, there is no neighbors of $A$ belonging to $V(F'')$.
	Let $A:=\{\alpha_1,..., \alpha _{2k}\}$.
	Note that for each $h\in [2k]$, we have $\sigma(\alpha_h u)=0$ and $\sigma(\alpha_h x_j)=1$.
	
	It is easy to see that there exist pairwise disjoint sets $N_h\subseteq N(\alpha_h)$ of size $2k$ for $h\in [2k]$.
	Then there exists $M_h\subseteq N_h$ such that $|M_h|= k-1$ and $\sigma(x\alpha_h)$ is the same for all $x\in M_h$.
	This gives $2k$ monochromatic copies of $K_{1,k-1}$ with centers $\alpha_h$ for $h\in [2k]$.
	At least $k$ of them (say with centers $\alpha_h$ for $h\in [k]$) have the same color.
	If this color is 0, these copies together with $\{u\alpha_h: h\in [k]\}$ and $F'$ give a monochromatic $J_k$ with color 0 and center $u$.
	Otherwise, this color is 1.
	Then these copies together with $\{x_j\alpha_h: h\in [k]\}$ and $F''$ give a monochromatic $J_k$ with color 1 and center $u$.
	This proves $\Tr(100k)\to J_k$. \qed

	\section{Monochromatic bistar}
	
	In this section we prove $\Tr(6k+30)\to B_k$. We first establish the following lemma.
	
	\begin{lem}\label{ll1}
		Let $uvwu$ be the triangle bounding the outer face of $\Tr(k+10)$.
		Let $\sigma: E(\Tr(k+10))\to \{0,1\}$ such that $|\{\sigma(ux): x\in V(\Tr(k+10))\}|=1$ and there is no monochromatic $C_4$ containing $u$.
		Then $\Tr(k+10))$ contains a monochromatic $B_k$.
	\end{lem}
	\pf Without loss of generality, let $\sigma(uv)=0$. Let $x_0:=w$ and, for $i\in [6]$,
	let $x_i\in V(\Tr(i))\setminus V(\Tr(i-1)$ such that $x_i$ is inside $uvx_{i-1}u$.
	
	Since $\Tr(k+10)$ has no monochromatic $C_4$ containing $u$, we see that $|\{0\leq i\leq 6:\sigma(vx_i)=0\}|\le 1$.
	So there exists some $i\in \{0,1,2,3,4\}$ such that
	$\sigma(vx_i)=\sigma(vx_{i+1})=\sigma(vx_{i+2})=1$.
	We have either $\sigma(x_ix_{i+1})=1$ or $\sigma(x_{i+1}x_{i+2})=1$;
	as otherwise $ux_ix_{i+1}x_{i+2}u$ is a monochromatic $C_4$ of color 0 and containing $u$, a contradiction.
	We consider two cases.
	
	\medskip
	\noindent {\it Case} 1.  $\sigma(x_ix_{i+1})=\sigma(x_{i+1}x_{i+2})$.
	
	In this case, we have $\sigma(x_ix_{i+1})=\sigma(x_{i+1}x_{i+2})=1$.
	So $x_ix_{i+1}x_{i+2}vx_i$ is a monochromatic $C_4$ of color 1.
	Let $y_0:=x_{i+1}$ and for $\ell\in [k+1]$, let $y_\ell\in
	V(\Tr(i+1+\ell))\setminus V(\Tr(i+\ell))$ such that $y_\ell$ is inside $uy_{\ell-1}x_iu$.
	Similarly let $z_0:=x_{i+1}$ and for $\ell\in [k+1]$, let $z_\ell\in V(\Tr(i+2+\ell))\setminus V(\Tr(i+1+\ell))$
	such that $z_\ell$ is inside $uz_{\ell-1}x_{i+2}u$.
	
	Since $\Tr(k+10)$ has no monochromatic $C_4$ containing $u$,
	this shows that $|\{\ell\in [k+1]: \sigma(x_iy_\ell)=0\}|\le 1$ and $|\{\ell\in [k+1]:\sigma(x_{i+2}z_\ell)=0\}|\le 1$.
	Therefore, there exist $Y\subseteq \{y_\ell: \ell\in [k+1]\}$ and $Z\subseteq \{z_\ell: \ell\in [k+1]\}$ such that
	$|Y|=|Z|=k$, $\sigma(yx_i)=1$ for each $y\in Y$ and $\sigma(zx_{i+2})=1$ for each $z\in Z$.
	Hence, $\Tr(k+10)$ has two monochromatic $K_{1,k}$s of color 1 with centers $x_i,x_{i+1}$ and leave sets $Y, Z$, respectively.
	These two $K_{1,k}$s together with $vx_ix_{i+1}x_{i+2}v$ form a monochromatic $B_k$ of color 1.
	
	\medskip
	
	\noindent {\it Case} 2. $\sigma(x_ix_{i+1})\ne \sigma(x_{i+1}x_{i+2})$.
	
	Without loss of generality, let $\sigma(x_ix_{i+1})=0$ and $\sigma(x_{i+1}x_{i+2})=1$.
	Let $y\in V(\Tr(i+2))\setminus V(\Tr(i+1))$ be
	inside $ux_ix_{i+1}u$. Because $\sigma(uy)=0$ and $\Tr(k+10)$ has no monochromatic $C_4$ containing $u$,
	$\sigma(yx_i)=\sigma(yx_{i+1})=1$.
	Therefore, $yx_{i+1}vx_iy$ is a monochromatic $C_4$ of color 1.
	Let $y_0:=y$ and, for $\ell\in [k+1]$, let $y_\ell\in
	V(\Tr(i+2+\ell))\setminus V(\Tr(i+1+\ell))$ such that $y_\ell$ is inside $uy_{\ell-1}x_iu$.
	Let $z_0:=y$ and, for $\ell\in [k+1]$, let $z_\ell\in V(\Tr(i+2+\ell))\setminus V(\Tr(i+1+\ell))$
	such that $z_\ell$ is inside $uz_{\ell-1}x_{i+1}u$.
	
	The remaining proof is similar as in Case 1.
	We observe that $|\{\ell\in [k+1]: \sigma(x_iy_\ell)=0\}|\le 1$ and $|\{\ell\in [k+1]:\sigma(x_{i+1}z_\ell)=0\}|\le 1$.
	Therefore, there exist $Y\subseteq \{y_\ell: \ell\in [k+1]\}$ and $Z\subseteq \{z_\ell: \ell\in [k+1]\}$ such that
	$|Y|=|Z|=k$, $\sigma(yx_i)=1$ for $y\in Y$, and $\sigma(zx_{i+1})=1$ for $z\in Z$.
	Hence, $\Tr(k+10)$ has two monochromatic $K_{1,k}$s of color 1 with centers $x_i,x_{i+1}$ and leave sets $Y, Z$, respectively.
	These two $K_{1,k}$s together with $yx_{i+1}vx_iy$ form a monochromatic $B_k$ of color 1.
	This proves Lemma \ref{ll1}.
	\qed

	\medskip

	We are ready to prove $\Tr(6k+30)\to B_k$.
	Let $\sigma: E(\Tr(6k+30)\to \{0,1\}$.
	By Lemma \ref{astu}, the copy of $\Tr(16)$ with the same outer face as of $\Tr(6k+30)$ contains a monochromatic $C_4$, say $u_1u_2u_3u_4u_1$ of color 1.
	For each $i\in \{1,3\}$, let $v_iw_i$ be an edge in $\Tr(18)$ such that $u_iv_iw_iu_i$ is a triangle inside $u_1u_2u_3u_4u_1$.
	Note that $u_iv_iw_iu_i$ bounds the outer face of a $\Tr(6k+12)$.
	Let $A_i$ be the family of all vertices $x\in N(u_i)$ inside $u_iv_iw_iu_i$ and satisfying $\sigma(xu_i)=1$.
	If $|A_1|\geq k$ and $|A_3|\geq k$, then together with the monochromatic 4-cycle $u_1u_2u_3u_4u_1$, it is easy to form a monochromatic $B_k$ of color 1.
	
	Hence by symmetry, we may assume that $|A_1|<k$.
	Then there exists an edge $vw$ in $\Tr(18+k)$ such that $u_1vwu_1$ bounds an inner face of $\Tr(18+k)$ and
	$\sigma(u_1x)=0$ for all $x\in N(u_1)$ in the closed disc bounded by $u_1vwu_1$.
	We may assume that the induced subgraph contained in the closed disc bounded by $u_1vwu_1$
	has a monochromatic $C_4$ say $u_1xyzu_1$ (as otherwise, it contains a $B_k$ by Lemma~\ref{ll1}).
	Furthermore, we have $\{x,y,z\}\subseteq V(\Tr(2k+28))$.
	
	Let $\{p_0,q_0\}\subseteq V(\Tr(2k+29))\setminus V(\Tr(2k+28))$ such that both
	$xyp_0x$ and $yzq_0y$ bound two inner faces of $\Tr(2k+29)$.
	For $\ell\in [3k]$, let $p_\ell\in V(\Tr(2k+29+\ell))\setminus V(\Tr(2k+28+\ell))$ such that
	$p_\ell$ is inside $xp_{\ell-1}yx$.
	Similarly, for $\ell\in [3k]$, let $q_\ell\in V(\Tr(2k+29+\ell))\setminus V(\Tr(2k+28+\ell))$ such that $q_{\ell}$ is inside $yq_{\ell-1}zy$.
	Moreover, let
	$$B_1:=\{p\in N(x): p \mbox{ is inside } xp_0yx  \mbox{ and } \sigma(xp)=0 \},$$
	$$B_2:=\{q\in N(z): q \mbox{ is inside } yq_0zy  \mbox{  and } \sigma(zq)=0\}.$$

	If $|B_1|\geq k$ and $|B_2|\geq k$, we can find two monochromatic $K_{1,k}$s of color 0, one inside $xp_0yx$ rooted at $x$ and one inside
	$yq_0zy$ rooted at $z$; these two $K_{1,k}$s and $u_1xyzu_1$ form a monochromatic $B_k$ of color 0.
	So we may assume, without loss of generality, that $|B_1|<k$.
	
	Let $C:=\{\ell\in [3k]: \sigma(yp_\ell)=0 \}$. We claim $|C|< k$.
	Suppose on the contrary that $|C|\ge k$.
	Then there is a monochromatic $K_{1,k}$ with root $y$ and leaves in $C$ of color 0.
	Since $\sigma(u_1p)=0$ for all $p\in N(u_1)$ inside $u_1vw$,
	there is also a monochromatic $K_{1,k}$ with root $u_1$ and leaves inside $u_1xyzu_1$ of color 0.
	Now these two $K_{1,k}$s and $u_1xyzu_1$ form a monochromatic $B_k$ of color 0.
	
	So $|B_1|<k$ and $|C|<k$. Then there exist $p_{h},p_{s}$ with $h,s\in [3k]$ such that
	$\sigma(p_hx)=\sigma(p_hy)=\sigma(p_{s}x)=\sigma(p_{s}y)=1$.
	Because $xp_0p_1x$ bounds an inner face of $\Tr(2k+30)$, it also bounds the outer face of a $\Tr(4k)$.
	As $|B_1|<k$, there exists a monochromatic $K_{1,k}$ of color 1 with the root $x$ and $k$ leavers inside $xp_0p_1x$.
	Similarly, as $|C|<k$, there exists a monochromatic $K_{1,k}$ of color 1 with root $y$ and $k$ leavers inside $yp_0p_1y$.
	Now these two $K_{1,k}$s and the 4-cycle $xp_hyp_sx$ form a monochromatic $B_k$ of color 1.
	This proves that $\Tr(6k+30)\to B_k$ and thus completes the proof of Theorem \ref{radius2}.  \qed

\end{document}